\documentclass[12pt]{amsart}
\usepackage{amssymb}
\usepackage[shortlabels]{enumitem}
\usepackage[all]{xy}
\usepackage{xcolor}
\usepackage[margin=1in]{geometry} 
\usepackage{hyperref}
\usepackage{eucal}
\hypersetup{bookmarksdepth=2}
\hypersetup{colorlinks=true}
\hypersetup{linkcolor=blue}
\hypersetup{citecolor=blue}
\hypersetup{urlcolor=blue}

\makeatletter
\@addtoreset{equation}{section}
\makeatother

\numberwithin{equation}{section}
\newtheorem{theorem}[equation]{Theorem}

\newtheorem{proposition}[equation]{Proposition}
\newtheorem{lemma}[equation]{Lemma}
\newtheorem{corollary}[equation]{Corollary}

\theoremstyle{definition}
\newtheorem{remark}[equation]{Remark}

\newcommand{\bA}{\mathbf{A}}

\newcommand{\cC}{\mathcal{C}}

\newcommand{\fD}{\mathfrak{D}}

\newcommand{\cM}{\mathcal{M}}

\newcommand{\bN}{\mathbf{N}}

\newcommand{\bO}{\mathbf{O}}

\newcommand{\cP}{\mathcal{P}}

\newcommand{\bS}{\mathbf{S}}

\newcommand{\fS}{\mathfrak{S}}

\newcommand{\bV}{\mathbf{V}}

\newcommand{\arxiv}[1]{\href{http://arxiv.org/abs/#1}{{\tiny\tt arXiv:#1}}}

\newcommand{\DOI}[1]{\href{http://doi.org/#1}{\color{purple}{\tiny\tt DOI:#1}}}
\newcommand{\defn}[1]{\emph{#1}}

\let\ul\underline
\renewcommand{\phi}{\varphi}

\DeclareMathOperator{\Ind}{Ind}
\DeclareMathOperator{\Spec}{Spec}

\DeclareMathOperator{\Hom}{Hom}

\DeclareMathOperator{\Sym}{Sym}

\newcommand{\ulambda}{\ul{\lambda}}
\renewcommand{\Vec}{\mathrm{Vec}}
\newcommand{\GL}{\mathbf{GL}}

\newcommand{\umu}{\ul{\smash{\mu}}}
\newcommand{\unu}{\ul{\smash{\nu}}}

\title[Invariant theory for linearly oligomorphic groups]{The fundamental theorems of invariant theory \\ for linearly oligomorphic groups}
\date{July 9, 2026}

\author{Alessandro Danelon}
\address{Department of Mathematics, University of Michigan, Ann Arbor, MI, USA}
\email{\href{mailto:adanelon@umich.edu}{adanelon@umich.edu}}
\urladdr{\url{https://public.websites.umich.edu/~adanelon/}}

\author{Andrew Snowden}
\thanks{AS was supported by NSF grant DMS-2301871.}
\address{Department of Mathematics, University of Michigan, Ann Arbor, MI, USA}
\email{\href{mailto:asnowden@umich.edu}{asnowden@umich.edu}}
\urladdr{\url{http://www-personal.umich.edu/~asnowden/}}

\begin{document}

\begin{abstract}
In recent work, Harman and the second author introduced some new infinite dimensional algebraic groups that generalize the classical groups. In this paper, we establish versions of the first and second fundamental theorems of invariant theory for them.
\end{abstract}

\maketitle
\tableofcontents

\section{Introduction}

Let $V$ be a finite dimensional vector space over an algebraically closed field $k$ of characteristic~0 equipped with a non-degenerate symmetric bilinear form $\langle, \rangle$. The function
\begin{displaymath}
q_{i,j} \colon V^m \to k, \qquad (v_1, \ldots, v_m) \mapsto \langle v_i, v_j \rangle
\end{displaymath}
is invariant under orthogonal group $\bO(V)$. The first fundamental theorem of invariant theory states that these functions generate the ring of $\bO(V)$-invariant polynomial functions on $V^m$. The second fundamental theorem describes the algebraic relations between these functions; in particular, if $\dim{V} \ge m$ then the $q_{i,j}$ (for $i \le j$) are algebraically independent. These theorems, and their analogs for other classical groups, have been extremely influential in the development of invariant theory.

Recently, Harman and the second author \cite{homoten} introduced a class of (typically infinite dimensional) algebraic groups called \defn{linearly oligomorphic groups}, and constructed examples analogous to the classical groups. In this paper, we establish versions of the fundamental theorems for them. We describe the simplest case here.

Let $n$ be a positive integer. A \defn{symmetric $n$-space} is a vector space $V$ equipped with a symmetric multi-linear form $\omega \colon V^n \to k$. By \cite[Theorem~A]{homoten}, there is a universal homogeneous symmetric $n$-space $(V, \omega)$ of countable dimension. ``Universal'' means that any finite dimensional symmetric $n$-space $W$ embeds into $V$, and ``homogeneous'' means that any two embeddings from $W$ differ by an automorphism of $V$. Remarkably, $(V, \omega)$ is unique up to isomorphism.

Let $G$ be the automorphism group of $(V, \omega)$. Homogeneity ensures that $G$ is large: for instance, $G$ acts transitively on vectors of norm~1. In fact, by \cite[Theorem~B]{homoten}, $G$ is linearly oligomorphic, which precisely formulates the manner in which $G$ is large. We view $G$ as a close cousin of the infinite orthogonal group. This perspective is bolstered by our main result:

\begin{theorem} \label{mainthm}
Let $(V, \omega)$ and $G$ be as above. For $1 \le i_1 \le \cdots \le i_n \le m$, define
\begin{displaymath}
q_{i_1,\ldots,i_n} \colon V^m \to k, \qquad (v_1, \ldots, v_m) \mapsto \omega(v_{i_1}, \ldots, v_{i_n}).
\end{displaymath}
These functions freely generate the ring of $G$-invariant symmetric forms on $V^m$.
\end{theorem}

See \S \ref{ss:forms} for the precise definition of the ring of symmetric forms. As mentioned above, this is the simplest case of our main theorem. The general case allows for multiple forms which need not be symmetric.

There is one other result we mention. Let $\bS_{\lambda}$ be the Schur functor associated to a partition $\lambda$, and regard $\bS_{\lambda}$ as taking values in the category of sets. We show that any natural transformation $\bS_{\lambda} \to \bS_{\mu}$ is algebraic, in a certain sense. See \S \ref{ss:auto} for the precise statement, and a more general result.

\subsection{Notation}

We list the most important notation.
\begin{description}[align=right,labelwidth=2.25cm,leftmargin=!]
\item[ $k$ ] the coefficient field (algebraically closed of characteristic~0)
\item[ $\fS_n$ ] the symmetric group of order $n$
\item[ $\bS_{\lambda}$ ] the Schur functor associated to the partition $\lambda$
\item[ $S^{\lambda}$ ] the Specht module associated to the partition $\lambda$
\item[ $\bA^{\ulambda}$ ] the basic $\GL$-variety associated to the tuple $\ulambda$
\item[ $\cC_{\ulambda}$ ] the category of finite $\ulambda$-spaces
\item[ $\cP(V)$ ] the ring of homogeneous forms on $V$
\end{description}

\section{Background}

\subsection{The ring of symmetric forms} \label{ss:forms}

Let $V$ be a vector space. A \defn{symmetric $d$-form} on $V$ is a function $V \to k$ that has the form $v \mapsto \omega(v, \ldots, v)$ for some symmetric multi-linear form $\omega \colon V^d \to k$. For example, when $d=2$ we recover the notion of quadratic form. We write $\cP_d(V)$ for the space of such forms; this is naturally identified with the linear dual of $\Sym^d{V}$. We let $\cP(V)=\bigoplus_{d \ge 0} \cP_d(V)$, which is naturally a ring; we call it the \defn{ring of symmetric forms}.

Suppose $V$ has basis $\{e_i\}_{i \in I}$. Let $\{x_i\}_{i \in I}$ be the dual functionals, i.e., $x_i \colon V \to k$ is the linear map satisfying $x_i(e_j)=\delta_{i,j}$. Then an element of $\cP_d(V)$ can be represented by a (perhaps infinite) $k$-linear combination of degree $d$ monomials in the $x_i$'s. In other words, $\cP(V)$ is naturally isomorphic to the inverse limit, in the category of graded rings, of the polynomial rings $k[x_i]_{i \in J}$ as $J$ runs over the finite subsets of $I$.

\subsection{Tensor spaces}

For a vector space $V$ and a partition $\mu$, a \defn{$\mu$-form} on $V$ is a linear map $\bS_{\mu}(V) \to k$ where $\bS_{\mu}$ denotes the Schur functor. Note that if $\mu=(d)$ then a $\mu$-form can be identified with a symmetric $d$-form, i.e., an element of $\cP_d(V)$.

Let $\ulambda=[\lambda_1, \ldots, \lambda_r]$ be a tuple of partitions; we often simply call $\ulambda$ a \emph{tuple}. A \defn{$\ulambda$-space} is a pair $(V, \ul{\omega})$, where $V$ is a vector space and $\ul{\omega}=(\omega_1, \ldots, \omega_r)$ is a tuple where $\omega_i$ is a $\lambda_i$-form on $V$. A \emph{map} of $\ulambda$-spaces $(W, \ul{\eta}) \to (V, \ul{\omega})$ is an injective linear map $i \colon W \to V$ such that $i^*(\omega_i)=\eta_i$ for each $1 \le i \le r$. We say that a $\ulambda$-space $V$ is \defn{finite} or \defn{countable} if $\dim{V}$ is. We let $\cC_{\ulambda}$ (resp.\ $\hat{\cC}_{\ulambda}$) denote the category of finite (resp.\ countable) $\ulambda$-spaces. Note that $\ul{\omega}$ can be identified with a linear map $\bS_{\ulambda}(V) \to k$, where $\bS_{\ulambda}=\bigoplus_{i=1}^r \bS_{\lambda_i}$; we will sometimes use this perspective.

Let $V$ be a countable $\ulambda$-space. We say that $V$ is \defn{universal} if every finite $\ulambda$-space embeds into $V$. We say that $V$ is \defn{homogeneous} if whenever $W$ is a finite $\ulambda$-space and $\alpha, \beta \colon W \to V$ are two embeddings, there exists an automorphism $\sigma$ of $V$ such that $\beta=\sigma \circ \alpha$. We say that the tuple $\ulambda$ is \defn{pure} if each $\lambda_i$ is non-empty. The following result is fundamental to this paper.

\begin{theorem}[{\cite[Theorem~A]{homoten}}]
Suppose $\ulambda$ is a pure tuple. Then there exists a countable universal homogeneous $\ulambda$-space. Moreover, any two are isomorphic.
\end{theorem}

We will be concerned with the automorphism group of the countable universal homogeneous $\ulambda$-space. As stated in the introduction, this group is very large: indeed, it is linearly oligomorphic by \cite[Theorem~B]{homoten}. As we will not need this result, we do not recall the definition of linearly oligomorphic.

\subsection{From $\lambda$-forms to symmetric forms} \label{ss:convert}

Let $\omega$ be a $\lambda$-form on $V$ with $\vert \lambda \vert =d$. We now describe how to convert $\omega$ into a symmetric $d$-form on $V^m$. Put $U=k^m$. Given $\alpha \in \bS_{\lambda}(U)^*$, let $\alpha \odot \omega$ be the composition
\begin{displaymath}
\xymatrix@C=3em{
\Sym^d(U \otimes V) \ar[r] &
\bS_{\lambda}(U) \otimes \bS_{\lambda}(V) \ar[r]^-{\alpha \otimes \omega} & k }
\end{displaymath}
where the first map is the natural projection from the Cauchy decomposition. We thus have a natural map
\begin{displaymath}
\bS_{\lambda}(U)^* \to \cP_d(U \otimes V), \quad \alpha \mapsto \alpha \odot \omega.
\end{displaymath}
We can convert this map into an element
\begin{displaymath}
[\omega]_U \in \bS_{\lambda}(U) \otimes \cP_d(U \otimes V).
\end{displaymath}
Note that $\alpha \odot \omega$ is the pairing of $\alpha$ with $[\omega]_U$. We often think of $[\omega]_U$ as a tuple of $\dim{\bS_{\lambda}(U)}$ symmetric $d$-forms on $V^m$. This interpretation depends on the choice of a basis of $\bS_{\lambda}(U)$, but this choice will be irrelevant in our applications.

Let $\ulambda=[\lambda_1, \ldots, \lambda_r]$ be a tuple, and let $(V, \ul{\omega})$ be a $\ulambda$-space. Let $e_i$ be the dimension of $\bS_{\lambda_i}(U)$, and let $\umu$ be the tuple $[(d_1)^{e_1}, \ldots, (d_r)^{e_r}]$ where $d_i = \vert \lambda_i \vert$; here the notation $(d_i)^{e_i}$ means that the partition $(d_i)$ is repeated $e_i$ times. Write $[\ul{\omega}]_U$ for the concatenation of the tuples $[\omega_1]_U, \ldots, [\omega_r]_U$. Then $(V, [\ul{\omega}]_U)$ is a $\umu$-space. This construction defines a functor $\cC_{\ulambda} \to \cC_{\umu}$. Write $\ell(\lambda)$ for the number of rows in the partition $\lambda$.

\begin{proposition} \label{prop:full}
If $m \ge \ell(\lambda_i)$ for each $i$ then the above functor is fully faithful.
\end{proposition}

\begin{proof}
The functor is faithful by construction. Suppose $f$ is a symmetric $d$-form on $V^m$ and $x \in \bS_{\lambda}(U)$. We can then build a $\lambda$-form on $V$ as the composition
\begin{displaymath}
\xymatrix{
\bS_{\lambda}(V) \ar[r]^-{x \otimes -} &
\bS_{\lambda}(U) \otimes \bS_{\lambda}(V) \ar[r] &
\Sym^d(U \otimes V) \ar[r]^-f & k }
\end{displaymath}
The second map is the natural inclusion coming from Schur--Weyl duality. Denote this form by $x \ast f$. Now, if $\omega$ is a $\lambda$-form and $\alpha \in \bS_{\lambda}(U)^*$, then one readily verifies that
\begin{displaymath}
x \ast (\alpha \odot \omega) = \alpha(x) \cdot \omega.
\end{displaymath}
This shows that we can recover $\omega$ from $[\omega]_U$ provided $m \ge \ell(\lambda)$; the final condition is needed to ensure $\bS_{\lambda}(U)$ is non-zero. In particular, if $(W, \ul{\eta})$ and $(V, \ul{\omega})$  are $\ulambda$-spaces and $f \colon W \to V$ is an injective linear map such that $f^*([\ul{\eta}]_U)=f^*([\ul{\omega}]_U)$ then necessarily $f^*(\ul{\eta})=f^*(\ul{\omega})$. This exactly means that the functor is full.
\end{proof}

\subsection{$\GL$-varieties}

Let $\bV=\bigcup_{n \ge 1} k^n$ and let $\GL=\GL(\bV)$. A representation of $\GL$ is \defn{polynomial} if it decomposes into a sum of $\bS_{\lambda}(\bV)$'s. A \defn{$\GL$-algebra} is a commutative algebra equipped with an action of $\GL$ by algebra homomorphisms under which it forms a polynomial representation. A basic example of a $\GL$-algebra is $R_{\ulambda} = \Sym(\bS_{\ulambda}(\bV))$. A $\GL$-algebra is \defn{finitely $\GL$-generated} if it is a quotient of $R_{\ulambda}$ for some tuple $\ulambda$. A \defn{$\GL$-variety} is the spectrum of a reduced finitely $\GL$-generated $\GL$-algebra. We let $\bA^{\ulambda}=\Spec(R_{\ulambda})$. This is the basic example of a $\GL$-variety: indeed, by definition, every $\GL$-variety is a closed $\GL$-subvariety of some $\bA^{\ulambda}$. See \cite{polygeom} for additional details.

The category of polynomial representations of $\GL$ is equivalent to the category of polynomial functors, with $\bS_{\lambda}(\bV)$ corresponding to $\bS_{\lambda}$. If $M$ is a polynomial representation and $V$ is a vector space, we write $M\{V\}$ for the value of the polynomial functor associated to $M$ on $V$. If $X=\Spec(R)$ is a $\GL$-variety, we write $X\{V\}$ for $\Spec(R\{V\})$. Note that $\bA^{\ulambda}\{V\}$ is the linear dual of $\bS_{\ulambda}(V)$, i.e., the space of $\ulambda$-structures on $V$. For this reason, $\GL$-varieties are closely related to $\ulambda$-spaces. This connection is explored in detail in \cite{homoten2}.

\section{Calculation of invariants} \label{s:inv}

Let $\ulambda=[\lambda_1, \ldots, \lambda_r]$ be a pure tuple, let $(V, \ul{\omega})$ be the universal homogeneous $\ulambda$-space, and let $G$ be its automorphism group; this notation is in effect for the entirety of \S \ref{s:inv}. The purpose of this section is to compute the $G$-invariant functionals on tensor powers of $V$. We formulate the answer in two ways, and then give some consequences of the calculation.

\subsection{The initial calculation}

Put $d_i=\vert \lambda_i \vert$ for $1 \le i \le r$. For a partition $\mu$ of $n$, let $S^{\mu}$ denote the corresponding Specht module, which is an irreducible $\fS_n$-module. Given a functional $\alpha \in (S^{\lambda_i})^*$, we obtain a linear map
\begin{displaymath}
\xymatrix@C=3em{
V^{\otimes d_i} \ar[r] & S^{\lambda_i} \otimes \bS_{\lambda_i}(V) \ar[r]^-{\alpha \otimes \omega_i} & k, }
\end{displaymath}
where the first map is the natural projection coming from Schur--Weyl duality. This map is clearly $G$-invariant. Using the self-duality of the Specht module, we thus have a map
\begin{displaymath}
S^{\lambda_i} \to \Hom_G(V^{\otimes d_i}, k).
\end{displaymath}
More generally, if we have an expression $d=m_1d_1+\cdots+m_rd_r$, then we can tensor together maps of the above form to get a map
\begin{displaymath}
(S^{\lambda_1})^{\otimes m_1} \otimes \cdots \otimes (S^{\lambda_r})^{\otimes m_r} \to \Hom_G(V^{\otimes d}, k).
\end{displaymath}
The target above is naturally a representation of the symmetric group $\fS_d$, while the source is a representation of the subgroup\footnote{Here the wreath product $\fS_d \wr \fS_m$ is defined as the semi-direct product $\fS^d \rtimes \fS_m$. It is a subgroup of $\fS_{dm}$, with the $i$th copy of $\fS_d$ permuting the $i$th block of size $d$ in $[dm]$, and $\fS_m$ permuting the blocks.} $H(\ul{m})=(\fS_{d_1} \wr \fS_{m_1}) \times \cdots \times (\fS_{d_r} \wr \fS_{m_r})$. By Frobenius reciprocity, we thus get a map of $\fS_d$-representations
\begin{displaymath}
\beta_{\ul{m}} \colon E_{\ul{m}} \to \Hom_G(V^{\otimes d}, k), \qquad
E_{\ul{m}} = \Ind_{H(\ul{m})}^{\fS_d} \left( (S^{\lambda_1})^{\otimes m_1} \otimes \cdots \otimes (S^{\lambda_r})^{\otimes m_r} \right).
\end{displaymath}
We now take the direct sum over $\ul{m}$ to obtain a map
\begin{displaymath}
\beta_d \colon E_d \to \Hom_G(V^{\otimes d}, k), \qquad E_d = \bigoplus_{\ul{m}} E_{\ul{m}}.
\end{displaymath}
Here $\ul{m}$ runs over elements of $\bN^r$ such that $d=m_1d_1+\cdots+m_rd_r$.

\begin{proposition} \label{prop:inv1}
The map $\beta_d$ is an isomorphism.
\end{proposition}

\begin{proof}
This is a reformulation of results from \cite{homoten} and \cite{tcares}. Precisely, in \cite[\S 5.2]{tcares}, we introduced the downwards $\ulambda$-Brauer category $\fD(\ulambda)$, a generalization of the downwards Brauer category from \cite{infrank}. By \cite[Lemma~4.12]{homoten} (see also \cite[Remark~4.14]{homoten}), there is a natural isomorphism
\begin{displaymath}
\Hom_G(V^{\otimes n}, V^{\otimes m}) = \Hom_{\fD(\ulambda)}([n], [m]).
\end{displaymath}
The map $\beta_d$ is exactly this isomorphism when $n=d$ and $m=0$.
\end{proof}

\subsection{A reformulation}

It will be useful to reformulate Proposition~\ref{prop:inv1}. Let $\umu$ be a tuple and suppose $f \colon \bA^{\ulambda} \to \bA^{\umu}$ is a map of $\GL$-varieties. Then $f$ induces a $G$-equivariant map $\bA^{\ulambda}\{V\} \to \bA^{\umu}\{V\}$, under which $\ul{\omega}$ maps to a $G$-invariant functional $\bS_{\umu}(V) \to k$. We therefore have a natural map
\begin{displaymath}
\gamma_{\umu} \colon \Hom(\bA^{\ulambda}, \bA^{\umu}) \to \Hom_G(\bS_{\umu}(V), k)
\end{displaymath}
The following is the result we require:

\begin{proposition} \label{prop:inv2}
The map $\gamma_{\umu}$ is an isomorphism.
\end{proposition}

\begin{proof}
It suffices to treat the case where $\umu$ consists of a single partition $\mu$. Let $d=\vert \mu \vert$. By Proposition~\ref{prop:inv1}, we have an isomorphism
\begin{displaymath}
\Hom_G(\bS_{\mu}(V), k) \cong \Hom_{\fS_d}(S^{\mu}, E_d).
\end{displaymath}
We now apply Schur--Weyl duality, which we regard as an equivalence between the category of $\fS_d$-modules and the category of degree $d$ polynomial functors. By definition, Schur--Weyl duality takes the Specht module $S^{\mu}$ to the Schur functor $\bS_{\mu}$. Using standard properties of Schur--Weyl duality, we see that $E(\ul{m})$ is taken to $\Sym^{m_1}(\bS_{\lambda_1}) \otimes \cdots \otimes \Sym^{m_r}(\bS_{\lambda_r})$. It follows that $E_d$ is taken to the degree $d$ piece of $\Sym(\bS_{\ulambda})$. We thus obtain isomorphisms
\begin{displaymath}
\Hom_{\fS_d}(S^{\mu}, E_d) = \Hom(\bS_{\mu}, \Sym(\bS_{\ulambda})) =
\Hom(\Sym(\bS_{\mu}), \Sym(\bS_{\ulambda})) = \Hom(\bA^{\ulambda}, \bA^{\mu}).
\end{displaymath}
Here the second $\Hom$ is taken in the category of polynomial functors, the third in the category of algebras internal to polynomial functors, and the fourth in the category of $\GL$-varieties. The first isomorphism above is Schur--Weyl duality, the second is the standard adjunction for $\Sym$, and the third comes from evaluating on $\bV$ and applying $\Spec$.

The above paragraph produces an isomorphism
\begin{displaymath}
\Hom(\bA^{\ulambda}, \bA^{\mu}) = \Hom_G(\bS_{\mu}(V), k)
\end{displaymath}
One could trace through the construction and show that this map coincides with $\gamma_{\mu}$. We will take a different approach. The above isomorphism shows that the source and target of $\gamma_{\mu}$ have the same dimension. It thus suffices to show that $\gamma_{\mu}$ is injective. Suppose that $f$ and $g$ are two maps $\bA^{\ulambda} \to \bA^{\mu}$ with $\gamma_{\mu}(f)=\gamma_{\mu}(g)$, i.e., $f(\omega)=g(\omega)$. Let $(W, \ul{\eta})$ be a finite $\ulambda$-space. Since $(V, \ul{\omega})$ is universal, there is an injective linear map $i \colon W \to V$ such that $i^*(\ul{\omega})=\ul{\eta}$. It follows that $f(\ul{\eta})=g(\ul{\eta})$. Thus $f$ and $g$ induce the same map $\bA^{\ulambda}\{W\} \to \bA^{\mu}\{W\}$ for all finite dimensional $W$, and so $f=g$, as required.
\end{proof}

\subsection{Automatic algebraicity} \label{ss:auto}

Let $\Vec$ be the category of finite dimensional vector spaces. We regard $\bA^{\ulambda}$ as a (contravariant) functor from $\Vec$ to the category of sets; note that this functor is simply $W \mapsto \bS_{\ulambda}(W)^*$. If $f \colon \bA^{\ulambda} \to \bA^{\umu}$ is a map of $\GL$-varieties then $f$ is a natural transformation between functors. We say that natural transformations of this form are \defn{algebraic}. We show:

\begin{proposition}
Any natural transformation $\bA^{\ulambda} \to \bA^{\umu}$ is algebraic.
\end{proposition}

\begin{proof}
Suppose $f \colon \bA^{\ulambda} \to \bA^{\umu}$ is a natural transformation. Given a (perhaps infinite dimensional) vector space $W$, we define a map
\begin{displaymath}
f_W \colon \bA^{\ulambda}\{W\} \to \bA^{\umu}\{W\}
\end{displaymath}
as follows. Let $\ul{\eta}$ be a $\ulambda$-structure on $W$. For any finite dimensional subspace $U$ of $W$, we can apply $f_U$ to $\ul{\eta} \vert_U$ to obtain a $\umu$-structure $f_U(\ul{\eta} \vert_U)$ on $U$. The naturality of $f$ implies these structures are compatible as $U$ varies, and thus glue to a $\umu$-structure $f_W(\ul{\eta})$ on $W$. One easily verifies that $W \mapsto f_W$ is a natural transformation between functors defined on all vector spaces.

From the above, it follows that we have a map
\begin{displaymath}
\operatorname{Nat}(\bA^{\ulambda}, \bA^{\umu}) \to \Hom_G(\bS_{\umu}(V), k), \qquad f \mapsto f(\ul{\omega}),
\end{displaymath}
where the domain is the set of all natural transformations. The reasoning in the second paragraph of the proof of Proposition~\ref{prop:inv2} shows that this map is injective. Since the restriction of this map to the set of algebraic natural transformations is bijective (Proposition~\ref{prop:inv2}), we see that all natural transformations are algebraic.
\end{proof}

\begin{remark}
By pre-composing with duality on $\Vec$, we see that natural transformations $\bA^{\ulambda} \to \bA^{\umu}$ coincide with natural transformations $\bS_{\ulambda} \to \bS_{\umu}$. Thus the above proposition determines the latter kind of natural transformations too.
\end{remark}

There is an alternate formulation of the above result worth pointing out. A \defn{$\Vec$-functor} $\cC_{\ulambda} \to \cC_{\umu}$ is a functor that does not change the underlying vector space, that is, it sends a $\ulambda$-space $(V, \ul{\omega})$ to a $\umu$-space $(V, \ul{\eta})$. One easily sees that such functors correspond to natural transformations $\bA^{\ulambda} \to \bA^{\umu}$. We thus find:

\begin{corollary}
Any $\Vec$-functor $\cC_{\ulambda} \to \cC_{\umu}$ is algebraic, i.e., induced from a map of $\GL$-varieties $\bA^{\ulambda} \to \bA^{\umu}$.
\end{corollary}

\subsection{Fixed points}

Given a $\GL$-variety $X$, let $\cM_{\ulambda}(X)$ denote the set of maps $\bA^{\ulambda} \to X$ of $\GL$-varieties. This set is naturally the $k$-points of a finite type scheme \cite[\S 2.6]{polygeom}, though this will not be relevant for us. Given a map $f \colon \bA^{\ulambda} \to X$, we can evaluate $f$ at $\ul{\omega}$ to obtain a $G$-invariant point in $X\{V\}$. We thus have a natural map
\begin{displaymath}
\iota \colon \cM_{\ulambda}(X) \to (X\{V\})^G
\end{displaymath}
We show:

\begin{proposition}
The map $\iota$ is an isomorphism.
\end{proposition}

\begin{proof}
When $X=\bA^{\umu}$, the assertion follows immediately from Proposition~\ref{prop:inv2}. We now treat a general $X$. Choose a cartesian square
\begin{displaymath}
\xymatrix{
X \ar[r] \ar[d] & 0 \ar[d] \\
\bA^{\umu} \ar[r]^f & \bA^{\unu} }
\end{displaymath}
One obtains such a square by realizing $X$ as a closed $\GL$-subvariety of $\bA^{\umu}$; the map $f$ is essentially recording the defining equations. Since $\cM_{\ulambda}(-)$ and $(-)^G$ are compatible with fiber products, we obtain cartesian squares
\begin{displaymath}
\xymatrix{
\cM_{\ulambda}(X) \ar[r] \ar[d] & 0 \ar[d] \\
\cM_{\ulambda}(\bA^{\umu}) \ar[r]^f & \cM_{\ulambda}(\bA^{\unu}) }
\qquad
\xymatrix{
(X\{V\})^G \ar[r] \ar[d] & 0 \ar[d] \\
(\bA^{\umu}\{V\})^G \ar[r]^f & (\bA^{\unu}\{V\})^G. }
\end{displaymath}
Moreover, $\iota$ gives a map between these squares. Since this map is an isomorphism at the vertices away from $X$, it is also an isomorphism at $X$, as required.
\end{proof}

The proposition shows that the $G$-fixed points on $X\{V\}$ are naturally the $k$-points of a finite type scheme.

\section{Strength and universality} \label{s:univ}

\subsection{Strength}

Let $R$ be a graded $k$-algebra with $R_0=k$ and $R_i=0$ for $i<0$. The \defn{strength} of a homogeneous element $f \in R_d$ is the minimum $s$ for which there exists an expression
\begin{displaymath}
f = \sum_{i=1}^s g_i h_i
\end{displaymath}
where $g_i$ and $h_i$ are homogeneous elements of degree $<d$. If no such expression exists, we say that $f$ has strength $\infty$. Note that the collection of finite strength elements is exactly the ideal $(R_+)^2$. More generally, we define the \emph{strength} of a collection of homogeneous elements to be the minimal strength of a non-trivial homogeneous linear combination.

We require the following two results related to strength.

\begin{theorem}[{\cite[Theorem~1.1]{stillman}}] \label{thm:stillman}
Let $V$ be a vector space and let $f_1, \ldots, f_r$ be homogeneous elements of $\cP(V)$ of infinite strength. Then $f_1, \ldots, f_r$ are algebraically independent.
\end{theorem}

\begin{theorem}[{\cite[Corollary~1.6]{KaZ}}] \label{thm:kaz}
Let $\ulambda=[(d_1), \ldots, (d_r)]$ with $d_i>0$ for each $i$, and let $(V, \ul{\omega})$ be a countable $\ulambda$-space. Then $(V, \ul{\omega})$ is universal if and only if $\omega_1, \ldots, \omega_r$ have infinite strength in $\cP(V)$.
\end{theorem}

We note that the paper \cite{KaZ} only works in finite dimension, but the above theorem follows easily from the results stated there. We also note that Theorem~\ref{thm:stillman} is actually an easy consequence of Theorem~\ref{thm:kaz}, but \cite[Theorem~1.1]{stillman} is a slightly stronger result.

\subsection{The main result}

Let $\ulambda=[\lambda_1, \ldots, \lambda_r]$ be a pure tuple, let $m \in \bN$ satisfy $m \ge \ell(\lambda_i)$ for each $i$, let $U=k^m$, let $d_i=\vert \lambda_i \vert$, let $e_i=\dim{\bS_{\lambda_i}(U)}$, and let $\umu=[(d_1)^{e_1}, \ldots, (d_r)^{e_r}]$ as in \S \ref{ss:convert}. The following is the main result of \S \ref{s:univ}, which connects the universality of a $\ulambda$-space to strength of the associated symmetric forms.

\begin{proposition} \label{prop:univ}
Let $(V, \ul{\omega})$ be a countable $\ulambda$-space. The following are equivalent:
\begin{enumerate}
\item $(V, \ul{\omega})$ is universal.
\item The associated $\umu$-space $(V, [\ul{\omega}]_U)$ is universal.
\item The forms $[\ul{\omega}]_U$ have infinite strength in the ring $\cP(V)$.
\end{enumerate}
\end{proposition}

Applying Theorem~\ref{thm:stillman}, we obtain the following corollary:

\begin{corollary} \label{cor:univ}
If $(V, \ul{\omega})$ is universal then the forms $[\ul{\omega}]_U$ are algebraically independent.
\end{corollary}

\begin{remark}
In \cite{BDDE}, a notion of strength for $\lambda$-forms is introduced and a direct analog of Theorem~\ref{thm:kaz} is established. Proposition~\ref{prop:univ} has a similar spirit to this result, but it is not exactly the same since it refers to the strength of the associated symmetric forms.
\end{remark}

We now begin working on the proof. If $\alpha$ and $\beta$ are $\lambda$-forms on $V$ and $W$ respectively, then we can make a $\lambda$-form on $V \oplus W$ via the composition
\begin{displaymath}
\bS_{\lambda}(V \oplus W) \to \bS_{\lambda}(V) \oplus \bS_{\lambda}(W) \to k,
\end{displaymath}
where the final map is the sum of $\alpha$ and $\beta$. We call this the \defn{orthogonal sum} of $\alpha$ and $\beta$.

\begin{lemma} \label{lem:str-sum}
Let $V_0$ be a vector space and let $f_0 \in \cP_d(V_0)$ be non-zero. Let $V=V_0^{\oplus 2s}$, and let $f \in \cP_d(V)$ be the orthogonal sum of $2s$ copies of $f_0$. Then $f$ has strength $\ge s$.
\end{lemma}

\begin{proof}
Let $w \in V_0$ be a vector such that $f_0(w) \ne 0$, let $W_0$ be the line spanned by $w$, and let $W=W_0^{\oplus 2s}$. The restriction of $f_0$ to $W_0$, in appropriate coordinates, is $x^d$, and so the restriction of $f$ to $W$ is $x_1^d+\cdots+x_{2s}^d$. It is well-known (see, e.g., \cite[Example~3.3]{BO}) that this form has strength $s$, and so $f$ has strength $\ge s$ as well.
\end{proof}

\begin{lemma} \label{lem:strength}
Let $\lambda$ be a non-empty partition, and assume $m \ge \ell(\lambda)$. Then for any $s \ge 0$, there exists a finite $\lambda$-space $(V, \omega)$ such that some component of $[\omega]_U$ has strength $\ge s$.
\end{lemma}

\begin{proof}
Let $(V_0, \omega_0)$ be a finite $\lambda$-space with $\omega_0 \ne 0$. Then $\alpha \odot \omega_0$ is non-zero for some $\alpha \in \bS_{\lambda}(U)^*$. Let $(V,\omega)$ be the orthogonal sum of $2s$ copies of $(V_0, \omega_0)$. Then $\alpha \odot \omega$ is an orthogonal sum of $2s$ copies of $\alpha \odot \omega_0$, and thus has strength $\ge s$ by Lemma~\ref{lem:str-sum}, as required.
\end{proof}

We are now ready to prove the proposition.

\begin{proof}[Proof of Proposition~\ref{prop:univ}]
(a) $\Rightarrow$ (c). We begin with two preliminary observations. First, if $(W, \ul{\eta})$ is a universal $\ulambda$-space then $(W, \eta_1)$ is a universal $\lambda_1$-space \cite[Lemma~3.4]{homoten}. Second, if $\lambda_1=\cdots=\lambda_s$ then $\GL_s$ acts on the category $\hat{\cC}_{\ulambda}$ by making a linear substituion in the first $s$ forms. This action clearly preserves universal spaces. In particular, if $(W, \ul{\eta})$ is a universal $\ulambda$-space then $(W, a_1 \eta_1+\cdots+a_s \eta_s)$ is a universal $\lambda_1$-space provided $(a_1, \ldots, a_s)$ is not the zero vector.

Now, suppose $(V, \ul{\omega})$ is universal but $[\ul{\omega}]_U$ has finite strength. We will derive a contradiction. We have a natural map
\begin{displaymath}
\phi \colon \bigoplus_{i=1}^r \bS_{\lambda_i}(U)^* \to \cP(U \otimes V)
\end{displaymath}
given by $\alpha \mapsto \alpha \odot \omega_i$ on the $i$th summand (see \S \ref{ss:convert}). It will be convenient to write this in a slightly different form. Write $\bS_{\ulambda}=\bigoplus_{\mu} W_{\mu} \otimes \bS_{\mu}$ where $W_{\mu}$ is a multiplicity space; thus $\dim{W_{\mu}}$ is the number of times $\mu$ occurs in $\ulambda$. Then $\phi$ takes the form
\begin{displaymath}
\phi \colon \bigoplus_{\mu} W_{\mu} \otimes \bS_{\mu}(U)^* \to \cP(U \otimes V).
\end{displaymath}
Let $I \subset \cP(U \otimes V)$ be the ideal of finite strength elements. It is clear that $I$ is $\GL(U)$-stable and that $\phi$ is $\GL(U)$-equivariant. Now, since $[\ul{\omega}]_U$ has finite strength, it follows that $\phi^{-1}(I)$ is non-zero. Since it is a $\GL(U)$-subrepresentation, it follows that there is some $\mu$ and some non-zero element $a \in W_{\mu}$ such that $a \otimes \bS_{\mu}(U)^*$ belongs to $\phi^{-1}(I)$. Reindexing if necessary, suppose that $\lambda_1=\cdots=\lambda_s=\mu$ and $\lambda_i \ne \mu$ for $i>s$; thus $W_{\mu}=k^s$. Put $\eta=a_1 \omega_1+\cdots+a_s \omega_s$. We thus see that all members of $[\eta]_U$ have finite strength. Since $(V, \eta)$ is a universal $\mu$-space, this contradicts Lemma~\ref{lem:strength}.

(b) $\Rightarrow$ (a). Suppose $(V, [\ul{\omega}]_U)$ is universal and let $(W, \ul{\eta})$ be a finite $\ulambda$-space. By assumption, we have an embedding $(W, [\ul{\eta}]_U) \to (V, [\ul{\omega}]_U)$ of $\umu$-spaces. Since the functor $\cC_{\ulambda} \to \cC_{\umu}$ is full (Proposition~\ref{prop:full}), it follows that there is an embedding $(W, \ul{\eta}) \to (V, \ul{\omega})$ of $\ulambda$-spaces. Thus $(V, \ul{\omega})$ is universal, as required.

(b) $\Leftrightarrow$ (c) is Theorem~\ref{thm:kaz}.
\end{proof}

\section{The main theorem} \label{s:rings}

Let $\ulambda=[\lambda_1, \ldots, \lambda_r]$ be a pure tuple, let $(V, \ul{\omega})$ be the countable universal homogeneous $\ulambda$-space, and let $G$ be its automorphism group. We now aim to compute the ring of $G$-invariant forms on $V^m$. We assume for simplicity that $m \ge \ell(\lambda_i)$ for each $i$. Put $U=k^m$, so that $V^m$ is identified with $U \otimes V$. Recall (\S \ref{ss:convert}) that we have a tuple of symmetric forms $[\ul{\omega}]_U$ on $U \otimes V$. These forms are clearly $G$-invariant. More canonically, we have a natural linear map
\begin{displaymath}
\bS_{\ulambda}(U)^* \to \cP(U \otimes V)^G
\end{displaymath}
which, by adjunction, induces a homomorphism
\begin{equation} \label{eq:ringmap}
\Sym(\bS_{\ulambda}(U)^*) \to \cP(U \otimes V)^G
\end{equation}
of graded rings. The following is our main result.

\begin{theorem} \label{thm:main}
The invariant ring $\cP(V^m)^G$ is the polynomial ring on the forms $[\ul{\omega}]_U$. More canonically, the map \eqref{eq:ringmap} is an isomorphism.
\end{theorem}

We note that if $r=1$ and $\lambda_1=(n)$ then this theorem specializes to Theorem~\ref{mainthm}. We begin with some lemmas. Let $\Phi_d$ be the polynomial functor $\Sym^d(U \otimes -)$.

\begin{lemma} \label{lem:cauchy}
For any degree $d$ polynomial functor $\Psi$, we have a natural isomorphism
\begin{displaymath}
\Hom(\Phi_d, \Psi) = \Psi(U^*).
\end{displaymath}
\end{lemma}

\begin{proof}
Write $\Psi = \bigoplus_{\lambda} W_{\lambda} \otimes \bS_{\lambda}$, where the sum is over partitions $\lambda$ of $d$ and $W_{\lambda}$ is a multiplicity space. By the Cauchy decomposition, we have $\Phi_d = \bigoplus_{\lambda} \bS_{\lambda}(U) \otimes \bS_{\lambda}$. We thus find
\begin{displaymath}
\Hom(\Phi_d, \Psi) = \bigoplus_{\lambda} \left( \bS_{\lambda}(U)^* \otimes W_{\lambda} \right) = \Psi(U^*),
\end{displaymath}
as required.
\end{proof}

\begin{lemma} \label{lem:main1}
$\cP(U \otimes V)^G$ is isomorphic to $\Sym(\bS_{\ulambda}(U^*))$ as a graded vector space.
\end{lemma}

\begin{proof}
First, observe that
\begin{displaymath}
\cP_d(U \otimes V)^G = \Hom_G(\Phi_d(V), k).
\end{displaymath}
Let $R$ be the $\GL$-algebra $\Sym(\Phi_d(\bV))$ and let $\bA(\Phi_d)$ be its spectrum; note that this is isomorphic to $\bA^{\umu}$ for an appropriate tuple $\umu$. By Proposition~\ref{prop:inv2}, we thus have
\begin{displaymath}
\cP_d(U \otimes V)^G = \Hom_{\text{$\GL$-var}}(\bA^{\ulambda}, \bA(\Phi_d)).
\end{displaymath}
Now, we have
\begin{displaymath}
\Hom_{\text{$\GL$-var}}(\bA^{\ulambda}, \bA(\Phi_d)) =
\Hom_{\text{$\GL$-alg}}(R, R_{\ulambda}) = \Hom_{\GL}(\Phi_d(\bV), R_{\ulambda}) = R_{\ulambda,d}\{U^*\},
\end{displaymath}
where $R_{\ulambda}=\Sym(\bS_{\ulambda}(\bV))$ and $R_{\ulambda,d}$ denotes its degree $d$ piece. In the final step we use Lemma~\ref{lem:cauchy}. The result thus follows by summing over $d$.
\end{proof}

We are now ready to prove the theorem.

\begin{proof}[Proof of Theorem~\ref{thm:main}]
The map \eqref{eq:ringmap} is injective since the forms $[\ul{\omega}]_U$ are algebraically independent (Corollary~\ref{cor:univ}). Since the graded pieces of the source and target have the same finite dimension (Lemma~\ref{lem:main1}), it is therefore an isomorphism.
\end{proof}

\begin{remark}
The automorphism groups of universal homogeneous tensor spaces are possibly the most important linearly oligomorphic groups, and the ones most similar to classical groups, but many others exists; see, e.g., \cite{homoten2} and \cite{biquad2}. It would be interesting to compute the invariant rings in other cases.
\end{remark}

\end{document}